\documentclass[oneside,english]{amsart}
\usepackage[T1]{fontenc}
\usepackage[latin9]{inputenc}
\usepackage{verbatim}
\usepackage{float}
\usepackage{amstext}
\usepackage{amsthm}
\usepackage{amssymb}
\usepackage{graphicx}
\usepackage[all]{xy}

\makeatletter
\numberwithin{equation}{section}
\numberwithin{figure}{section}
  \theoremstyle{plain}
  \newtheorem*{thm*}{\protect\theoremname}
\theoremstyle{plain}
\newtheorem{thm}{\protect\theoremname}
  \theoremstyle{definition}
  \newtheorem{defn}[thm]{\protect\definitionname}
  \theoremstyle{plain}
  \newtheorem{lem}[thm]{\protect\lemmaname}
  \theoremstyle{plain}
  \newtheorem{prop}[thm]{\protect\propositionname}
  \theoremstyle{definition}
  \newtheorem{example}[thm]{\protect\examplename}

\makeatother

\usepackage{babel}
  \providecommand{\definitionname}{Definition}
  \providecommand{\examplename}{Example}
  \providecommand{\lemmaname}{Lemma}
  \providecommand{\propositionname}{Proposition}
  \providecommand{\theoremname}{Theorem}
\providecommand{\theoremname}{Theorem}

\begin{document}
\global\long\def\dd{\textup{d}}
\global\long\def\EQ#1#2{\raisebox{.65ex}{#1}/\raisebox{-.65ex}{#2}}

\title{Schlesinger foliation for deformations of foliations.}

\author{Yohann Genzmer}
\begin{abstract}
In this article, we show that for any deformation of analytic foliations,
there exists a maximal analytic singular foliation on the space of
parameters along the leaves of which the deformation is integrable.
\end{abstract}

\maketitle
{\footnotesize{}This work was partially supported by ANR-13-JS01-0002-0.}{\footnotesize \par}

\section{Introduction and statements.}

The Painlevé $VI$ equation \cite{Boalch}, that motivates this work,
is a non-linear differential equation of order two. The leaves of
the induced foliation parametrize the integrable deformations of linear
Fuchsian systems over the four-punctured complex sphere. This a particular
case of the Schlesinger systems that parametrize the integrable deformations
of Garnier systems. In both case, a foliation on the space of parameters
of the deformation whose leaves parametrize the integrable deformations
is described. 

Our purpose is to highlight this phenomenon in a general context for
deformations of regular foliations. 
\begin{thm*}
\label{thm:Let--a}Let $\left(X^{m},\mathcal{F}^{n}\right)\xrightarrow{\pi}B^{p}$
be a proper deformation of analytic regular foliations with $B^{p}$
as space of parameters. Then there exists a unique analytic singular
foliation $\mathcal{H}$ on $B^{p}$ of maximal dimension among those
that \emph{integrates} the deformation. In particular, $\mathcal{H}$
satisfies the following property: for any leaf $L$ of $\mathcal{H},$
the restricted deformation 
\[
\left.\left(X^{m},\mathcal{F}^{n}\right)\right|_{\pi^{-1}\left(L\right)}\xrightarrow{\pi}L
\]
 is integrable.
\end{thm*}
The definition of \emph{a foliation integrating} a given deformation
will appear below. 

If we remove the maximality property, then the theorem becomes trivial.
Indeed, the foliation of $B^{p}$ by points satisfies its conclusion.
Moreover, for a generic deformation, the foliation $\mathcal{H}$
produced by the result above will be indeed the foliation by points.
Nevertheless, in view of the example mentioned in the introduction,
the foliation $\mathcal{H}$ deserves to be called \emph{the Schlesinger
foliation} of the deformation. Finally, the main theorem holds in
the real analytic class as well as in the complex one. 

It might be possible that along some exceptional curves transverse
to $\mathcal{H}$, the deformation is also integrable. Such a curve
has to be more or less \emph{isolated}: they cannot foliate the manifold
$B^{p}$ even locally. For instance, in the framework of the theory
of Fuchsian systems, consider the six parameters family defined by
\begin{equation}
\frac{\dd}{\dd z}=\frac{A_{1}}{z-u_{1}}+\frac{A_{2}}{z-u_{2}}+\frac{A_{3}}{z-u_{3}}\label{eq:11}
\end{equation}
where $A_{i}=\left(\begin{array}{cc}
0 & a_{i}\\
0 & 0
\end{array}\right)$. Since the matrices $A_{i}$ commute, the Schlesinger system \cite{Schlesinger}
reduces to 
\[
\frac{\partial A_{i}}{\partial u_{j}}=0,\qquad i,j=1,\ 2,\ 3.
\]
Thus, the Schlesinger foliation of this deformation is given by $\dd a_{1}=\dd a_{2}=\dd a_{3}=0.$
However, if $a_{1}=a_{2}=0$ and $a_{3}\neq0$, then (\ref{eq:11})
degenerates toward a system which is conjugated to 
\[
\frac{\dd}{\dd z}=\frac{\left(\begin{array}{cc}
0 & 1\\
0 & 0
\end{array}\right)}{z-u_{3}}.
\]
Therefore, at any point $\left(0,0,a_{3}\right)$ with $a_{3}\neq0$,
the family (\ref{eq:11}) is still integrable along the whole submanifold
$a_{1}=a_{2}=0$, which is not a leaf of the Schlesinger foliation.
Hence, in general, the leaf of $\mathcal{H}$ does not parametrize
the maximal submanifold along which the deformation is integrable.

The main theorem can be compared to the following result of Kiso \cite{Kiso}:
let $L$ be a Lie algebra of holomorphic vector fields on a complex
manifold $N\times M$ tangent to the projection on $M$. If $L$ is
simple and of finite dimension then there exists a maximal foliation
on $N$ along the leaves of which $L$ is integrable. In our context,
the Lie algebra underlying the foliation $\mathcal{F}^{n}$ is in
general not of finite dimension. Moreover, the fibration $\pi$ is
not trivial: the manifold supporting the foliation can also be deformed.

\bigskip{}

\section{Deformation of foliations and Kodaira-Spencer map.}

\subsection{Deformations of foliations. }

Few definitions below concern foliations with maybe a singular locus,
but the main theorem is stated only for regular deformations of regular
foliations. 

Let $X^{m}$ be an analytic manifold and $\Theta_{X^{m}}$ its tangent
sheaf. Let $E$ be a coherent subsheaf of $\Theta_{X^{m}}.$ The stalk
of $E$ at any point $p$ is finitely generated as $\mathcal{C}^{\omega}\left(X^{m}\right)_{p}$-module
- the analytic functions on $B^{p}$ - and we denote by $\textup{rank}_{p}\left(E\right)$
the minimal number of elements of a generating family. The \emph{singular
locus} of $E$ is defined by 
\[
\textup{Sing}\left(E\right)=\left\{ \left.p\in X^{m}\right|\dim_{\mathbb{R}}\left\{ \left.v\left(p\right)\right|v\in E_{p}\right\} <\textup{rank}_{p}\left(E\right)\right\} 
\]

Since $E$ is coherent, the singular locus is an analytic subset of
$X^{m}$ \cite{Mitera}. The sheaf is said \emph{regular} if $\textup{Sing}\left(E\right)=\emptyset$.
The integer 
\[
d=\max_{p\in X^{m}}\dim_{\mathbb{R}}\left\{ \left.v\left(p\right)\right|v\in E_{p}\right\} \leq m.
\]
is called the dimension of $E$ and $m-d$ its codimension. 
\begin{defn}
An analytic foliation on $X^{m}$ is a coherent subsheaf $\mathcal{F}$
of $\Theta_{X^{m}}$ which is \emph{integrable}, i.e, 
\[
\left[\mathcal{F}_{p},\mathcal{F}_{p}\right]\subset\mathcal{F}_{p}\quad\textup{for any }p\in X^{m}\setminus\textup{Sing}\left(\mathcal{F}\right).
\]
\end{defn}
\begin{defn}
A deformation of regular foliation, denoted by $\left(X^{m},\mathcal{F}^{n}\right)\xrightarrow{\pi}B^{p}$,
is the data of a proper submersion $\pi:X^{m}\to B^{p}$, where $X^{m}$
and $B^{p}$ are smooth manifold and of a regular foliation $\mathcal{F}^{n}$
tangent to the fibers of $\pi,$ i.e., $\mathcal{F}^{n}\subset\ker d\pi$.
The integer $m$ is the dimension of $X^{m}$, $p$ the dimension
of $B^{p}$ and $n$ the dimension of $\mathcal{F}^{n}$.
\end{defn}
The following lemma is a direct consequence of the Frobénius \cite{MR552968}
theorem and states that a deformation of foliation is locally trivial
in the total space.
\begin{lem}
\label{lem:Frob}Let $\left(X^{m},\mathcal{F}^{n}\right)\xrightarrow{\pi}B^{p}$
be a deformation of foliation and $x\in X^{m}.$ There exists an isomorphism
$\phi$ of deformations of foliation defined on an open neighborhood
$U\ni x$ such that the following diagram commutes 
\[
\xymatrix{\left(U,\left.\mathcal{F}^{n}\right|_{U}\right)\ar[rr]^{\phi\quad\qquad\quad}\ar[rd]^{\pi} &  & \left(\mathbb{R}^{m-p}\times B^{p},L^{n}\right)\ar[ld]^{\textup{pr}_{2}}\\
 & \pi\left(U\right)
}
.
\]
In the second term of the above commutative diagram, the foliation
$L^{n}$ of $\mathbb{R}^{m-p}\times B^{p}$ is given by the fibers
of the projection 
\[
\Pi:\left\{ \begin{array}{ccc}
\mathbb{R}^{m-p}\times B^{p} & \to & \mathbb{R}^{m-p-n}\times B^{p}\\
\left(x_{1},\cdots,x_{m-p},\tau\right) & \to & \left(x_{1},\cdots,x_{m-p-n},\tau\right)
\end{array}\right..
\]

\begin{figure}[H]
\centering{}\includegraphics[scale=0.5]{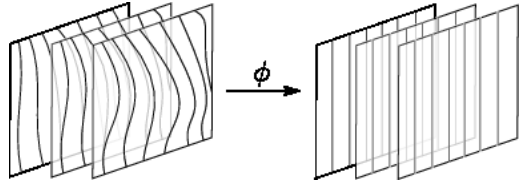}\caption{\label{fig:Local-trivialization-of}Local trivialization of a deformation
of foliation.}
\end{figure}
\end{lem}
\begin{defn}
\label{Let--a}Let $\left(X^{m},\mathcal{F}^{n}\right)\xrightarrow{\pi}B^{p}$
be a deformation of foliations. A foliation $\mathcal{H}$ of $B^{p}$
is said to \emph{integrate} the deformation if, for any point $\tau\in B^{p}\setminus\textup{Sing}\left(\mathcal{H}\right)$,
there exists a neighborhood $U\ni\tau$ and a regular foliation $\mathcal{G}$
in $\pi^{-1}\left(U\right)$ such that $\dim\mathcal{G}=n+\dim\mathcal{H}$
where $n$ is the dimension of $\mathcal{F}^{n}$ and 
\[
\left.\mathcal{F}^{n}\right|_{\pi^{-1}\left(U\right)}\subset\mathcal{G}\subset\pi^{*}\left.\mathcal{H}\right|_{U}.
\]

The deformation is said to be \emph{completely integrable} if and
only if the trivial foliation $\mathcal{H}=B^{p}$ integrates the
deformation.
\end{defn}
In particular, for any leaf $L$ of $\mathcal{H},$ the restricted
deformation 
\[
\pi:\left(\pi^{-1}\left(L\right),\left.\mathcal{F}^{n}\right|_{\pi^{-1}\left(L\right)}\right)\to L
\]
is completely integrable. In Figure (\ref{fig:Local-trivialization-of}),
the deformation is locally completely integrable: the foliation $\mathcal{H}$
has one leaf and the foliation $\mathcal{G}$ of the definition above
is given in the coordinates of Lemma \ref{lem:Frob} by the fibers
of the projection 

\[
p:\left\{ \begin{array}{ccc}
\mathbb{R}^{m-p}\times B^{p} & \to & \mathbb{R}^{m-p-n}\\
\left(x_{1},\cdots,x_{m-p},\tau\right) & \to & \left(x_{1},\cdots,x_{m-p-n}\right)
\end{array}\right..
\]

Our goal is to show the existence of a unique maximal integrating
foliation. 

\subsection{The sheaf of basic vector fields.}
\begin{defn}
A vector field $v$ is said to be basic for $\mathcal{F}$ if and
only if 
\[
\left[v,\mathcal{F}\right]\subset\mathcal{F}.
\]
\end{defn}
\begin{defn}
A vector field $v$ is said to be projectable if and only if there
exists a of vector field $w$ on $B^{p}$ such that the following
diagram commutes
\[
\xymatrix{X^{m}\ar[r]^{\pi}\ar[d]^{v} & B^{p}\ar[d]^{w}\\
TX^{m}\ar[r]^{d\pi} & TB^{p}
}
.
\]

It is said to be vertical if $w=0,$ i.e., $d\pi\left(v\right)=0.$ 
\end{defn}
We denote by $\mathcal{B}^{\pi}$ the sheaf of basic and projectable
vector fields and $\mathcal{B}^{0}$ the subsheaf of basic and vertical
vector fields. The sheaf $\mathcal{B}^{0}$ is a sheaf of modules
over the ring of local first integrals of $\mathcal{F}^{n}.$ In the
coordinates given by Lemma \ref{lem:Frob}, a section of the quotient
sheaf $\EQ{\mbox{\ensuremath{\mathcal{B}^{0}}}}{\mbox{\ensuremath{\mathcal{F}^{n}}}}$
is written 
\[
\sum_{i=1}^{m-p-n}a\left(x_{1},\ldots,x_{m-p-n},\tau\right)\frac{\partial}{\partial x_{i}}.
\]
Hence, $\EQ{\mbox{\ensuremath{\mathcal{B}^{0}}}}{\mbox{\ensuremath{\mathcal{F}^{n}}}}$
is a free module. It will be more convenient to work with a $\mathcal{C}^{\omega}\left(X^{m}\right)$-module,
which is why, we will consider the product $\EQ{\mbox{\ensuremath{\mathcal{B}^{0}}}}{\mbox{\ensuremath{\mathcal{F}^{n}}}}\otimes\mathcal{C}^{\omega}\left(X^{m}\right).$
It is also a locally free sheaf over $\mathcal{C}^{\omega}\left(X^{m}\right)$. 

\subsection{The basic Kodaira-Spencer map. }

Let $w$ be a vector field on $B^{p}$ defined on a small open set
$W$ of $B^{p}$. According to Lemma \ref{lem:Frob}, there exists
a covering $\left\{ U_{i}\right\} _{i\in I}$ of a neighborhood in
$X^{n}$ of $\pi^{-1}\left(W\right)$ such that the deformation is
trivialized on any $U_{i}$ by some conjugacy $\phi_{i}.$ For any
$U_{i},$ the vector field $v_{i}=d\phi_{i}^{-1}\left(0,w\right)$
is a section of $\mathcal{B}^{\pi}$ on $U_{i}$ that projects on
$w.$ Considering the image of the family $\left\{ v_{i}\otimes1-v_{j}\otimes1\right\} _{ij}$
in $H^{1}\left(\pi^{-1}\left(W\right),\EQ{\mbox{\ensuremath{\mathcal{B}^{0}}}}{\mbox{\ensuremath{\mathcal{F}^{n}}}}\otimes\mathcal{C}^{\omega}\left(X^{m}\right)\right)$
, we obtain a $\mathcal{C}^{\omega}\left(B^{p}\right)-$morphism of
sheaves 
\[
\partial\mathcal{F}^{n}:\Theta_{B^{p}}\to R^{1}\pi_{*}\left(\EQ{\mbox{\ensuremath{\mathcal{B}^{0}}}}{\mbox{\ensuremath{\mathcal{F}^{n}}}}\otimes\mathcal{C}^{\omega}\left(X^{m}\right)\right)
\]

which is called the \emph{basic Kodaira-Spencer map} of the deformation.
Notice that this is not the standard Kodaira-Spencer map as defined
in \cite{Kod3} since, it does not measure the infinitesimal directions
of triviality, but the infinitesimal directions of integrability. 
\begin{lem}
\label{lem:-is-a}$\ker\partial\mathcal{F}^{n}$ is a foliation of
$B^{p}.$ 
\end{lem}
\begin{proof}
The sheaf $\EQ{\mbox{\ensuremath{\mathcal{B}^{0}}}}{\mbox{\ensuremath{\mathcal{F}^{n}}}}\otimes\mathcal{C}^{\omega}\left(X^{m}\right)$
is a coherent sheaf of $\mathcal{C}^{\omega}\left(X^{m}\right)$-modules.
Since $\pi$ is proper, $R^{1}\pi_{*}\left(\EQ{\mbox{\ensuremath{\mathcal{B}^{0}}}}{\mbox{\ensuremath{\mathcal{F}^{n}}}}\otimes\mathcal{C}^{\omega}\left(X^{m}\right)\right)$
is a coherent sheaf of $\mathcal{C}^{\omega}\left(X^{m}\right)_{B^{p}}-$modules
, and so is $\ker\partial\mathcal{F}^{n}$ \cite{grauertremmert}
. Suppose $w_{1}$ and $w_{2}$ belong to the kernel of $\partial\mathcal{F}^{n}$.
By construction, there exist two families $\left\{ v_{i}^{\epsilon}\right\} _{i}$
, $\epsilon=1,2$ of sections of $\mathcal{B}^{\pi}$ such that for
any $i$ and $\epsilon$, the vector field $v_{i}^{\epsilon}$ projects
on $w_{\epsilon}$ and such that for any $i,\ j$ and $\epsilon$,
\[
v_{i}^{\epsilon}-v_{j}^{\epsilon}=t_{ij}^{\epsilon}\in\mathcal{F}^{n}.
\]

The difference of their Lie brackets is written
\begin{eqnarray*}
\left[v_{i}^{1},v_{i}^{2}\right]-\left[v_{j}^{1},v_{j}^{2}\right] & = & \left[v_{j}^{1},t_{ij}^{2}\right]-\left[v_{i}^{2},t_{ij}^{1}\right]+\left[t_{ij}^{1},t_{ij}^{2}\right]
\end{eqnarray*}
Since $\left[\mathcal{B}^{\pi},\mathcal{F}^{n}\right]\subset\mathcal{F}^{n}$
and since the map $d\pi$ commutes with the Lie bracket, one has 
\[
\partial\mathcal{F}^{n}\left(\left[w_{1},w_{2}\right]\right)=\overline{\left\{ \left[v_{i}^{1},v_{i}^{2}\right]\otimes1-\left[v_{j}^{1},v_{j}^{2}\right]\otimes1\right\} _{i,j}}=0.
\]
 So $\ker\partial\mathcal{F}^{n}$ is integrable and thus is a foliation. 
\end{proof}

\section{Integrating foliations of deformations. }

Consider the trivial projection $\mathbb{R}^{m-p}\times\mathbb{R}^{p}\to\mathbb{R}^{p}$
where the source is foliated by $L^{n}$ given by the fibers of 
\[
\left\{ \begin{array}{ccc}
\mathbb{R}^{m-p}\times\mathbb{R}^{p} & \to & \mathbb{R}^{m-p-n}\times\mathbb{R}^{p}\\
\left(x_{1},\cdots,x_{m-p},\tau\right) & \to & \left(x_{1},\cdots,x_{m-p-n},\tau\right).
\end{array}\right.
\]
Let $\left\{ T_{i}\right\} _{i=1,\cdots,l}$ be an involutive family
of germs of vector fields inducing a germ of regular foliation of
dimension $l$ in $\left(\mathbb{R}^{p},\tau\right)$. The family
\[
\left\{ \partial_{x_{m-p-n+1}},\cdots,\ \partial_{x_{m-p}},T_{1},\cdots,T_{l}\right\} 
\]
where $T_{i}$ is seen as vector field in $\mathbb{R}^{m-p}\times\mathbb{R}^{p}$
is involutive and defines an integrating dimension $n+l$ foliation
$\mathcal{G}$. Notice that $\mathcal{G}$ is not unique but does
depend only on the \emph{basic part }of $T_{i}$: for any family of
vector fields $\left\{ X_{i}\right\} _{i=1..l}$ tangent to $\left\{ \partial_{x_{m-p-n+1}},\cdots,\ \partial_{x_{m-p}}\right\} $,
the distribution 
\[
\left\{ \partial_{x_{m-p-n+1}},\cdots,\ \partial_{x_{m-p}},T_{1}+X_{1},\cdots,T_{l}+X_{l}\right\} 
\]
induces the same foliation $\mathcal{G}$. This remark is the key
of the proof of the proposition below. 
\begin{prop}
\label{prop:Let--a}Let $\left(X^{m},\mathcal{F}^{n}\right)\xrightarrow{\pi}B^{p}$
be a deformation of foliation and $w$ be a vector field in $B^{p}$
defined near $\tau$ with $w\left(\tau\right)\neq0.$ The two following
properties are equivalent:

\begin{enumerate}
\item the foliation induced by $w$ integrates the deformation. 
\item $\partial\mathcal{F}^{n}\left(w\right)=0$.
\end{enumerate}
\end{prop}
\begin{proof}
Suppose that there exists a regular foliation $\mathcal{G}$ of dimension
$n+1$ in $X^{m}$ such that $\left.\mathcal{F}^{n}\right|_{\pi^{-1}\left(U\right)}\subset\mathcal{G}\subset\pi^{*}\mathcal{H}$
where $\mathcal{H}$ is the foliation of $B^{p}$ induced by $w$
in a neighborhood $U\ni\tau$. Applying the classical Frobénius result
to $\mathcal{G}$ and straightening locally the fibration $\pi$ yield
a covering $\left\{ U_{i}\right\} _{i\in I}$ of $\pi^{-1}\left(U\right)$
and a family of conjugacies $\left\{ \phi_{i}\right\} _{i\in I}$
such that the following diagrams commute
\[
\xymatrix{\left(U_{i},\left.\mathcal{G}\right|_{U_{i}}\right)\ar[rr]^{\phi_{i}\quad\qquad\quad}\ar[rd]^{\pi} &  & \left(\mathbb{R}^{m-p}\times B^{p},\mathbb{R}^{m-p}\times\mathcal{H}\right)\ar[ld]^{\textup{pr}_{2}}\\
 & \pi\left(U_{i}\right)
}
\]

By construction, the vector field $v_{i}=d\phi_{i}^{-1}\left(0,w\right)$
is basic for $\mathcal{F}^{n}$ and projects on $w$. Thus $\partial\mathcal{F}^{n}\left(w\right)=\overline{\left\{ v_{i}\otimes1-v_{j}\otimes1\right\} _{ij}}.$
Moreover, $v_{i}-v_{j}$ is vertical and tangent to $\mathcal{G}.$
Thus, it is also tangent to $\mathcal{F}^{n}.$ Hence, $\partial\mathcal{F}^{n}\left(w\right)=0$.

Now, suppose that $\partial\mathcal{F}^{n}\left(w\right)=0$. For
a covering $\left\{ U_{i}\right\} _{i\in I}$ of a neighborhood of
$\pi^{-1}\left(U\right)$, there exists a family of projectable basic
vector fields $\left\{ v_{i}\right\} _{i\in I}$ , $v_{i}\in\Theta_{X^{n}}\left(U_{i}\right)$
such that each $v_{i}$ projects on $w$ and such that $v_{i}-v_{j}$
is tangent to $\mathcal{F}^{n}$. In the local coordinates given by
Lemma \ref{lem:Frob}, the vector field $v_{i}$ is written 
\[
v_{i}=\sum_{i=1}^{m-p-n}a_{i}\left(x_{1},\cdots,x_{m-p-n},\tau\right)\frac{\partial}{\partial x_{i}}+\sum_{i=m-p-n+1}^{m-p}a_{i}\left(x,\tau\right)\frac{\partial}{\partial x_{i}}+w.
\]
Since 
\[
\left[v_{i}-\sum_{i=m-p-n+1}^{m-p}a_{i}\left(x,\tau\right)\frac{\partial}{\partial x_{i}},\partial_{x_{k}}\right]=0
\]
for $k=m-p-n+1,\ldots,m-p$, the family of vector fields $\left\{ \partial_{x_{m-p-n+1}},\cdots,\ \partial_{x_{m-p}},v_{i}\right\} $
is involutive and defines a local regular integrating foliation $\mathcal{G}_{i}$
of dimension $n+1$. The foliation $\mathcal{G}_{i}$ depends only
on the basic part of $v_{i}$ and $v_{i}-v_{j}$ is tangent to $\mathcal{F}^{n}$,
thus the family $\left\{ \mathcal{G}_{i}\right\} _{i\in I}$ can be
glued in a global foliation that integrates the deformation over $U$. 
\end{proof}
Now, we can prove the main result of this article. 
\begin{thm*}
Let $\left(X^{m},\mathcal{F}^{n}\right)\xrightarrow{\pi}B^{p}$ be
a deformation of foliations. Then, there exists a unique singular
analytic foliation $\mathcal{D}$ on $B^{p}$ of maximal dimension
among those that integrate the deformation. 
\end{thm*}
\begin{proof}
Let us consider the sub-sheaf $\mathcal{D}$ of $\Theta_{B^{p}}$
defined by $v\in\mathcal{D}$ if and only the foliation defined by
$v$ integrates the deformation. According to Proposition \ref{prop:Let--a}
and Lemma \ref{lem:-is-a}, $\mathcal{D}$ is a foliation. By construction,
if $\mathcal{D}$ has the property of the theorem, it will be of maximal
dimension for that property. Let $d$ be its dimension and consider
a point $p$ in its regular locus. The sheaf $\mathcal{D}$ is locally
around $p$ spanned by a family of exactly $d$ non-vanishing vector
fields $\mathcal{D}\left(U\right)=\left\langle w_{1},\cdots,w_{d}\right\rangle $,
$U\ni p$. Now, for some covering $\left\{ U_{i}\right\} _{i\in I}$
of a neighborhood of $\pi^{-1}\left(U\right)$, there exists a family
of projectable basic vector fields $\left\{ v_{i,k}\right\} _{i\in I,\ k=1\ldots d}$,
$v_{i,k}\in\Theta_{X^{n}}\left(U_{i}\right)$ such that $v_{i,k}$
projects on $w_{k}$ and $v_{i,k}-v_{j,k}$ is tangent to $\mathcal{F}^{n}$.
In the local coordinates given by Lemma \ref{lem:Frob}, the family
of vector fields 
\[
\left\{ \partial_{x_{m-p-n+1}},\cdots,\ \partial_{x_{m-p}},v_{i,1},\ldots,v_{i,d}\right\} 
\]
is involutive and defines a family of local integrating foliations
$\left\{ \mathcal{G}_{i}\right\} _{i\in I}$ of dimension $n+d$ that
can be glued. Finally, we obtain a global regular integrating foliation
that integrates the deformation over $U$. 
\end{proof}
\begin{example}
\emph{Linear foliations on complex tori of dimension 2.}

A complex torus of dimension $2$ is given by a lattice $\Lambda$
written 
\[
\Lambda=\mathbb{Z}e_{1}\oplus\mathbb{Z}e_{2}\oplus\mathbb{Z}e_{3}\oplus\mathbb{Z}e_{4}
\]
where $\left\{ e_{i}=\left(\begin{array}{c}
e_{i1}\\
e_{i2}
\end{array}\right)\right\} _{i=1\cdots4}$ is a $\mathbb{R}-$free family of four vectors in $\mathbb{C}^{2}$.
The complex torus $\mathbb{C}^{2}/\Lambda$ is endowed with a linear
foliation given by the closed $1-$form 
\[
\omega_{\alpha}=\dd x+\alpha\dd y
\]
with $\alpha\in\mathbb{P}^{1}\left(\mathbb{C}\right).$ The linear
foliations on complex tori of dimension 2 form a $9-$dimensional
family of foliations whose space of parameters is $\mathcal{U}\times\mathbb{P}^{1}$.
Here, $\mathcal{U}$ is the real Zariski open set in $\mathbb{C}^{8}$
given by the equation 
\[
\det\left(\begin{array}{cccc}
\Re e_{1} & \Re e_{2} & \Re e_{3} & \Re e_{4}\\
\Im e_{1} & \Im e_{2} & \Im e_{3} & \Im e_{4}
\end{array}\right)\neq0.
\]
An explicit computation shows that the Schlesinger foliation of this
deformation is given by the closed system 
\begin{equation}
\mathcal{H}:\ \left\{ \begin{array}{c}
d\left(\frac{e_{11}+e_{12}\alpha}{e_{21}+e_{22}\alpha}\right)=0\\
d\left(\frac{e_{11}+e_{12}\alpha}{e_{31}+e_{32}\alpha}\right)=0\\
d\left(\frac{e_{11}+e_{12}\alpha}{e_{41}+e_{42}\alpha}\right)=0
\end{array}\right..\label{eq:3}
\end{equation}
It is regular on the space of parameters $\mathcal{U}\times\mathbb{P}^{1}$.
This result can also be seen as follows: the representation of monodromy
of $\omega_{\alpha}$ on $\mathbb{C}^{2}/\Lambda$ computed on the
transversal line $x=0$ is written
\[
\pi_{1}\left(\mathbb{C}^{2}/\Lambda\right)\simeq\mathbb{Z}^{4}\to\textup{Aut}\left(\mathbb{CP}^{1}\right):\ \gamma_{i}\to\left(\begin{array}{c}
y\to y+e_{i1}+e_{i2}\alpha\end{array}\right),\ i=1,2,3,4
\]

A change of global coordinates $y\to ay+b$ on the transversal line
$x=0$ acts on the monodromy the following way 
\[
\begin{array}{c}
y\to y+a\left(e_{i1}+e_{i2}\alpha\right)\end{array},\ i=1,2,3,4
\]
Hence, the conjugacy class of the representation of monodromy is constant
along a deformation if and only if the quotients $\frac{e_{11}+e_{12}\alpha}{e_{i1}+e_{i2}\alpha},\ i=2,3,4$
are constant, which is precisely the condition \ref{eq:3}. Besides
that, in this example, any leaf of the Schlesinger foliation parametrizes
the maximal locus of integrability for any point in the space of parameters. 
\end{example}
The foliation $\mathcal{H}$ can be compactified as an algebraic foliation
$\overline{\mathcal{H}}$ on $\mathbb{CP}^{8}\times\mathbb{CP}^{1}$
and the fibration 
\[
\begin{array}{ccc}
\mathbb{CP}^{8}\times\mathbb{CP}^{1} & \to & \mathbb{CP}^{6}\\
\left(\left\{ e_{ij}\right\} _{i,j},\alpha\right) & \to & \left(e_{11}+e_{12}\alpha,e_{12},e_{22},e_{32},e_{41},e_{42}\right)
\end{array}
\]
is transverse to $\overline{\mathcal{H}}$ except over the critical
locus $S=\left\{ \left(e_{11}+e_{12}\alpha\right)e_{42}=0\right\} \subset\mathbb{CP}^{6}.$
This fibration has compact fibers. Thus, according to the Ehresmann's
theorem, the foliation $\overline{\mathcal{H}}$ has the \emph{Painlevé
property} with respect to that fibration as defined in \cite{Oka2}.
Since, the Painlevé VI equation has also the same property, it is
natural to address the following problem:
\begin{quote}
\emph{Consider an algebraic deformation of an algebraic foliation.
Has the Schlesinger foliation the Painlevé property for an adapted
fibration ? }
\end{quote}
\bibliographystyle{plain}
\bibliography{/home/genzmer/ownCloud/Article/Biblio/Bibliographie}
\bigskip{}

\begin{minipage}{0.49\linewidth} Y. Genzmer\\ {\scriptsize Institut de Math\'ematiques de Toulouse\\ Universit\'{e} Paul Sabatier \\ 118 route de Narbonne \\ 31062 Toulouse cedex 9, France.\\ yohann.genzmer@math.univ-toulouse.fr} \end{minipage}
\end{document}